\documentclass[12pt]{amsart}
\makeatletter

\@addtoreset{equation}{section}
\makeatother
\usepackage[margin=3cm]{geometry}
\newenvironment{nouppercase}{%
  \renewcommand{\uppercasenonmath}[1]{}}{}
\usepackage{amsmath,amssymb,amsthm,color}
\usepackage{mathrsfs} 
\usepackage[T1]{fontenc}
\usepackage{ytableau}
\usepackage[all]{xy}
\allowdisplaybreaks[1]
\theoremstyle{definition}
\newtheorem{theorem}[equation]{Theorem}

\newtheorem*{conj*}{Conjecture}
\newtheorem*{theorem*}{Theorem}
\newtheorem{remark}[equation]{Remark}

\newtheorem{lemma}[equation]{Lemma}

\begin{document}

\title{An interpolation of the generalized duality formula for the Schur multiple zeta values to complex functions}
\author{Maki Nakasuji}
\address{Department of Information and Communication Science, Faculty of Science, Sophia
University, 7-1 Kioi-cho, Chiyoda-ku, Tokyo, 102-8554, Japan}
\email{nakasuji@sophia.ac.jp}
\author[Y. Ohno]{Yasuo Ohno}
\address{Mathematical Institute, Tohoku University, Sendai 980-8578, Japan}
\email{ohno.y@m.tohoku.ac.jp}
\author[W. Takeda]{Wataru Takeda}
\address{Department of Applied Mathematics, Tokyo University of Science,
1-3 Kagurazaka, Shinjuku-ku, Tokyo 162-8601, Japan.}
\email{w.takeda@rs.tus.ac.jp}

\begin{nouppercase}
\maketitle
\end{nouppercase}
\begin{abstract}
One of the important research subjects in the study of multiple zeta functions is to clarify the linear relations and functional equations among them.
The Schur multiple zeta functions are a generalization of the multiple zeta functions of Euler-Zagier type.
Among many relations, the duality formula and its generalization are important families for both Euler-Zagier type and Schur type multiple zeta values. 
In this paper, following the method of previous works for multiple zeta values of Euler-Zagier type, we give an interpolation of the sums in the generalized duality formula, called Ohno relation, for Schur multiple zeta values. Moreover, we prove that the Ohno relation for Schur multiple zeta values is valid for complex numbers.
\end{abstract}
\section{Introduction}
For positive integers $r$, $k_1, k_2, \ldots, k_r$ with $k_r\ge2$,  
a multiple zeta value of Euler-Zagier type
is defined by
\[
 \zeta(k_1,\ldots,k_r)=\sum_{1\le n_1<\cdots< n_r}\frac{1}{n_1^{k_1}\cdots n_r^{k_r}},
 \]
 where the summation runs over all the size $r$ sets of ordered positive integers.
One can confirm that the above series converges for $r$-tuples $(k_1,\ldots,k_r)$ of positive integers with $k_r\ge2$. These $r$-tuples $(k_1,\ldots,k_r)$ are called admissible.
{Many $\mathbb Q$-linear relations among multiple zeta values are known. Especially, the duality formula and its generalization are important relations.}
To state the generalized duality formula, we denote a string $\underbrace{1,\ldots,1}_r$ of $1$'s by $\{1\}^r$.
Then, for an admissible index
\begin{equation}
\label{index}
{\bf k}=(\{1\}^{a_1-1}, b_1+1, \{1\}^{a_2-1}, b_2+1, \ldots, \{1\}^{a_m-1}, b_m+1)
\end{equation}
with $a_1, b_1, a_2, b_2, \cdots, a_m, b_m\in {\mathbb Z}_{\geq 1}$, the following index is called a {\it dual} index of ${\bf k}$:
\[{\bf k^{\dagger}}=(\{1\}^{b_m-1}, a_m+1, \{1\}^{b_{m-1}-1}, a_{m_1}+1, \ldots, \{1\}^{b_1-1}, a_1+1).
\]
The generalized duality formula, called Ohno relation in some literatures, can then be described as follows:
\begin{theorem}[The generalized duality formula. \cite{Oho}] \label{ohno}
 For any $\ell\in {\mathbb Z_{\geq 0}}$ and
any admissible index ${\bf k}=(k_1, \ldots, k_r)$, and its dual index ${\bf k^{\dagger}}=(k^{\dagger}_1., \ldots, k^{\dagger}_s)$,
\begin{equation}
\label{ohnosummzv11}
    \sum_{\substack{\varepsilon_1+\cdots+\varepsilon_r=\ell\\\varepsilon_i\geq 0}}\zeta(k_1+\varepsilon_1, \ldots, k_r+\varepsilon_r)=
\sum_{\substack{\varepsilon'_1+\cdots+\varepsilon'_s=\ell\\
\varepsilon_i'\geq 0}} \zeta(k^{\dagger}_1+\varepsilon'_1, \ldots, k^{\dagger}_s+\varepsilon'_s).
\end{equation}
\end{theorem}
In Theorem \ref{ohno}, when $\ell=0$, we obtain the duality formula for multiple zeta values of Euler-Zagier type.
We may write the left-hand side of \eqref{ohnosummzv11} as ${\mathcal O}({\bf k} : \ell)$ and call $\mathcal{O}$-sum, then \eqref{ohnosummzv11} can be written as \begin{equation}
\label{ohnosummzv1}
    \mathcal{O}({\bf k}:\ell)=\mathcal{O}({\bf k}^\dagger:\ell).
\end{equation}
In \cite{no}, the authors generalized Theorem \ref{ohno} to the Schur multiple zeta values under some conditions.
In the following, we review their setup:

For any partition $\lambda$, i.e., a non-increasing sequence $( \lambda_1, \ldots, \lambda_n)$ of { positive} integers,
we associate the Young diagram
$D_{\lambda}=\{(i, j)\in {\mathbb Z}^2 ~|~ 1\leq i\leq n, 1\leq j\leq \lambda_i\}$ depicted as a collection of square boxes with the $i$-th row having $\lambda_i$ boxes.
For a partition $\lambda$, a Young tableau $T=(t_{ij})$ of shape $\lambda$ over a set $X$ is obtained by filling the boxes of $D_{\lambda}$ with $t_{ij}\in X$.
 We denote by $T_\lambda(X)$ the set of all Young tableaux of shape $\lambda$ over $X$ and denote by $SSYT_{\lambda}$ the set of semi-standard Young tableaux {$(t_{ij})\in T_\lambda(\mathbb N)$ which satisfies} the condition of weakly increasing from left to right in each row $i$,
and strictly increasing from top to bottom in each column $j$.
{\color{blue}Let $\lambda=(\lambda_1, \ldots, \lambda_r)$ and $\mu=(\mu_1, \ldots, \mu_s)$ be two partitions such that $\lambda_i\geq \mu_i$ for all $i$ and $r\ge s$, and let $\delta=\lambda/\mu$ be a partition of skew shape. 
Then we define $D_\delta=D_\lambda\setminus D_\mu$ and sets of their fillings $T_\delta(X)$, $SSYT_\delta$ in the same way as above. 
Then, for a given set ${ \pmb k}=(k_{ij})\in T_{\delta}(\mathbb{Z})$ of a tableau index, 
{\it Schur multiple zeta values} of {shape} $\delta$ are defined as
\[
\zeta_{\delta}({ \pmb k})=\sum_{M\in SSYT_{\delta}}\frac{1}{M^{ \pmb k}},
\]
where $M^{ \pmb k}=\displaystyle{\prod_{(i, j)\in D_{\delta}}m_{ij}^{k_{ij}}}$ for $M=(m_{ij})\in SSYT_{\delta}$. 
The function $\zeta_\delta(\pmb k)$ absolutely converges in \[
 W_{\delta}
=
\left\{{\pmb k}=(k_{ij})\in T_{\delta}(\mathbb{Z})\,\left|\,
\begin{array}{l}
 \text{$k_{ij}\ge 1$ for all $(i,j)\in D_{\delta} \setminus C_{\delta}$} \\[3pt]
 \text{$k_{ij}\ge2$ for all $(i,j)\in C_{\delta}$}
\end{array}
\right.
\right\},
\]
where $C_\delta$ is the set of all corners of $\delta$.
Here, we say that $(i,j)\in D_{\delta}$ is a corner of $\delta$ if $(i+1, j)\notin D_{\delta}$ and $(i, j+1)\notin D_{\delta}$;
for example, if $\delta=(4,3,3,2)\setminus(3,2,1)$, $C_\delta=\{(1,4),(3,3),(4,2)\}$.
In this article, we assume that all tableau indices of $\zeta_\delta$ are elements of $W_\delta$.} 
Nakasuji and Ohno \cite{no} defined a tableau which is ``dual'' to ${\pmb k}\in T_{\delta}({\mathbb Z})$.
First, we denote a finer piece of index $\{1\}^{a-1}, b+1$ as $A(a,b)$ and call it an {\it admissible piece}. Then, if we write $A_i=A(a_i, b_i)$
and $A_i^{\dagger}=A(b_i, a_i)$,  
the above
admissible index ${\bf k}$ and its dual ${\bf k}^{\dagger}$ can be written in terms of admissible pieces:
\[{\bf k}=(A(a_1, b_1), A(a_2, b_2), \ldots, A(a_m, b_m))=(A_1, A_2, \ldots, A_m)
\]
and
\[{\bf k^{\dagger}}=(A(b_m, a_m), A(b_{m-1}, a_{m-1}), \ldots, A(b_1, a_1))=(A_m^{\dagger}, A_{m-1}^{\dagger}, \ldots, A_1^{\dagger}).
\]
{\color{blue}We now write ${\pmb k}\in T_{\delta}(\mathbb Z)$ as
\begin{equation}\label{notationk}
{\pmb k}={\pmb k}^{\mathrm{col}}_1\cdots {\pmb k}^{\mathrm{col}}_{\lambda_1},
\end{equation}
where ${\pmb k}^{\mathrm{col}}_{j}$ is the $j$-th column tableau of ${\pmb k}$. For example, when $\lambda=(3, 2, 1)$ and
\[{\pmb k}=
\ytableausetup{boxsize=normal,aligntableaux=center}  
\begin{ytableau}
 k_{11}   & k_{12} & k_{13}\\
 k_{21}  & k_{22}\\
 k_{31}
\end{ytableau},
\]
then 
${\pmb k}^{\mathrm{col}}_{1}=
\begin{ytableau}
 k_{11} \\
 k_{21} \\
 k_{31}
\end{ytableau}
$,
${\pmb k}^{\mathrm{col}}_{2}=
\begin{ytableau}
 k_{12} \\
 k_{22} 
\end{ytableau}
$
and
${\pmb k}^{\mathrm{col}}_{3}=
\begin{ytableau}
 k_{13} 
\end{ytableau}
$.}
Let $T^{\mathrm{diag}}_\delta({\mathbb Z})=\{{\pmb k}\in T_{\delta}({\mathbb Z}) ~|~ k_{ij}=k_{pq} \; {\mathrm{if}}\; j-i=q-p\}$. 
 Let $I_{\delta}^D$ be the set of elements in $T_{\delta}^{\mathrm{diag}}({\mathbb Z})$ consisting of
admissible pieces such that the right side of the top element in each column is not $1$. 
For ${\pmb k}\in I_{\delta}^D$, in terms of admissible pieces, 
the row that has the topmost component is identified as the first row. 
In terms of admissible pieces, we can write as $A_{ij}$ the component in the $i$-th row and $j$-th column.
Note that
the component in the upper-right corner in the $j$-th row is $A_{1j}$
and that $A_{ij}=A_{k\ell}$ if $j-i=\ell-k$ when they are not empty. 
Further, we note that, in terms of tableau,
the top element in $A_{ij}$ and the bottom element in $A_{i(j+1)}$ are located side by side.

Let $\delta=\lambda/\mu$ be a partition of skew shape. For ${\pmb k}=(k_{ij})\in W_\delta({\mathbb Z}_{\geq 1})$,
${\pmb \varepsilon}=(\varepsilon_{ij})\in T_{\delta}({\mathbb Z}_{\geq 0})$, and $\ell\in {\mathbb Z_{\geq 0}}$, we denote by 
\[{\mathcal O}({\pmb k}: \ell)=\sum_{\substack{|{\pmb \varepsilon}|=\ell}} \zeta_{\delta} ({\pmb k}+{\pmb \varepsilon}),
\]
where ${\pmb k}+{\pmb \varepsilon}=(k_{ij}+\varepsilon_{ij})$ and $|{\pmb \varepsilon}|=\sum_{(i,j)\in D_\delta} \varepsilon_{ij}$.
For ${\pmb k}_{j}^{\mathrm{col}}\in I_{(1^n)}^D$ ($j_1\leq j\leq j_r$)'s,  we define
\[
{\mathcal O}({\pmb k}_{j_1}^{\mathrm{col}}\cdots{\pmb k}_{j_r}^{\mathrm{col}} : \ell) 
=\sum_{\ell_1+\ell_2+\cdots+\ell_r=\ell}
{\mathcal O}({\pmb k}_{j_1}^{\mathrm{col}} : \ell_1) \cdots {\mathcal O}({\pmb k}_{j_r}^{\mathrm{col}} : \ell_r).\]
\begin{theorem}[\cite{no}]\label{Thmohno}
Let $\lambda$ and $\mu$ be partitions and let $\delta=\lambda/\mu$. If ${\pmb k}^{\dagger}$ is the dual tableau of ${\pmb k} \in I_{\delta}^{\rm D}$ and $\ell\in {\mathbb Z_{\geq 0}}$, we have
\begin{equation}
\label{ohnosummzv}
{\mathcal O}({\pmb k}: \ell)={\mathcal O}({\pmb k}^{\dagger}: \ell).    
\end{equation}
\end{theorem}
We may regard \eqref{ohnosummzv} as a generalization of \eqref{ohnosummzv1}.
Identities \eqref{ohnosummzv1} and \eqref{ohnosummzv} are based on the addition of positive integers. On the other hand, in \cite{hmo}, Hirose, Murahara and Onozuka gave an interpolation of \eqref{ohnosummzv1} to complex functions.
For an admissible index ${\bf k}=(k_1,\ldots,k_r)$ and $s\in\mathbb{C}$ with $\Re(s)>-1$, they defined the function $I_{{\bf k}}(s)$, called Ohno function in \cite{ko}, by
\begin{align}\label{ohnofunctionzeta}
 I_{{\bf k}}(s)=\sum_{i=1}^{r}\sum_{0<n_1<\cdots<n_r} \frac{1}{n_1^{k_1}\cdots n_r^{k_r}}
  \cdot \frac{1}{n_i^{s}} \prod_{j\ne i} \frac{n_j}{ n_j-n_i }.
\end{align}
In \cite[Lemma 2.2]{hmo}, it is proved that if $s$ is a non-negative integer $m\in\mathbb{Z}_{\ge 0}$, the function $I_{{\bf k}}(s)$ is the same as $\mathcal{O}$-sum, that is,
\begin{align*}
I_{{\bf k}}(m)&=\sum_{\substack{ \varepsilon_1+\cdots+\varepsilon_r=m \\ \varepsilon_i\ge0\,(1\le i\le r) }}\zeta (k_1+\varepsilon_1,\ldots,k_r+\varepsilon_r)\\
&=\mathcal{O}({\bf k}:m).
\end{align*}
Thus, by Theorem \ref{ohno}, we have $I_{{\bf k}}(m)=I_{{\bf k}^\dagger}(m)$.
More generally, they gave an interpolation of the Ohno relation to complex numbers.
\begin{theorem}[\cite{hmo}]
  \label{interpolation}
 For an admissible index ${\bf k}$ and $s\in\mathbb{C}$, we have
 \begin{align*}
  I_{{\bf k}}(s)=I_{{\bf k}^\dagger}(s).
 \end{align*}
\end{theorem}
Subsequently, Kamano and Onozuka introduced two kinds of integral representations of (\ref{ohnofunctionzeta}):
\begin{theorem}[\cite{ko}]
For any admissible index ${\bf k}$ represented as (\ref{index})
  and $s\in\mathbb{C}$ with $\Re(s)>-1$,  we have
\begin{align*}
I_{{\bf k}}(s) &=\dfrac{1}{(a_1-1)!(b_1-1)!  \cdots (a_m-1)! (b_m-1)! \Gamma (s+1) }\\
    &\times \int_{0<t_1<\cdots < t_{2m}<1}  \dfrac{dt_1 \cdots dt_{2m}}{(1-t_1)t_2  \cdots (1-t_{2m-1})t_{2m} }\left( \log  \dfrac{t_2 \cdots t_{2m} }{t_1 \cdots t_{2m-1}}  \right)^{s}\\ 
	&\times \left(  \log  \dfrac{1-t_1}{1-t_2}   \right)^{a_1-1}
	\left(  \log \dfrac{t_3}{t_2}   \right)^{b_1-1}
	\cdots
	\left( \log  \dfrac{1-t_{2m-1}}{1-t_{2m}}   \right)^{a_{m}-1}
	\left( \log  \dfrac{1}{t_{2m}}   \right)^{b_{m}-1}.
\end{align*}
\end{theorem}
\begin{theorem}[\cite{ko}]
\label{integral}
For any admissible index ${\bf k}=(k_1,\ldots,k_r)$ and $s\in\mathbb{C}$ with $\max_{1\leq j\leq r}\{r-2j+2-(k_j+\cdots+k_r)\}<\Re(s)<0$,  we have
\[
I_{{\bf k}}(s)=-\frac{\sin(\pi s)}{\pi}\sum_{0<n_1<\cdots<n_r}\frac{1}{n_1^{k_1-1}\cdots n_r^{k_r-1}}\int_0^\infty\frac{w^{-s-1}}{(w+n_1)\cdots(w+n_r)}dw.
\]
\end{theorem}
In this paper, we generalize the integral representation given in Theorem \ref{integral} to the Schur multiple zeta values. In other words, we consider the function  
\[I_{{\pmb k}}(s)=-\frac{\sin(\pi s)}{\pi}\sum_{(n_{ij})\in SSYT_{\lambda/\mu}}\prod_{(i,j)\in D_{\lambda/\mu}}\frac{1}{n_{ij}^{k_{ij}-1}}\int_0^{\infty}w^{-s-1}\prod_{(i,j)\in D_{\lambda/\mu}}\frac{1}{w+n_{ij}}dw\]
 in Section \ref{section2} and show that this function actually interpolates $\mathcal{O}$-sum for the Schur multiple zeta values in Section \ref{section3}. Moreover, we prove
\begin{theorem}[Theorem \ref{dualityschur}]
Let $\lambda$ and $\mu$ be partitions. Put $\delta=\lambda/\mu$ and ${\pmb k} \in I_{\delta}^{\rm D}$
and let ${\pmb k}^{\dagger}$ be the dual tableau of ${\pmb k}$, for $s\in \mathbb C$ we have
\[
I_{\pmb k}(s)=I_{{\pmb k}^{\dagger}}(s).
\]
\end{theorem}

\section{Integral representation and series expansion}
\label{section2}
In this section, we make preparations for constructing the function $I_{{\pmb k}}(s)$ for Schur multiple zeta values.
As in introduction, taking Theorem \ref{integral} into account, proved by Kamano and Onozuka \cite[Theorem 1.6]{ko}, we can expect that $I_{\pmb k}(s)$ can be defined as follows: 
\[I_{\pmb k}(s)=-\frac{\sin(\pi s)}{\pi}\sum_{(n_{ij})\in SSYT_{\lambda/\mu}}\prod_{(i,j)\in D_{\lambda/\mu}}\frac{1}{n_{ij}^{k_{ij}-1}}\int_0^{\infty}w^{-s-1}\prod_{(i,j)\in D_{\lambda/\mu}}\frac{1}{w+n_{ij}}dw.\]
We first prove the following lemma for the calculation of this integral:
 \begin{lemma}
  \label{betas}
  For any positive integers $r, n$ and $s\in \mathbb C$ with $-r<\Re(s)<0$,
  \[\int_0^{\infty}\frac{w^{-s-1}}{(w+n)^r}dw=-\frac{\pi}{\sin(\pi s)}\frac{1}{n^{s+r}}\prod_{\ell=1}^{r-1} \frac{s+r-\ell}{r-\ell}.\]
  \end{lemma}
  \begin{proof}
  Changing the variable by $w=nv$ leads to
 \begin{align*}
     \int_0^{\infty}\frac{w^{-s-1}}{(w+n)^r}dw&=n^{-s-r}\int_0^{\infty}\frac{v^{-s-1}}{(v+1)^r}dv=\frac{1}{n^{s+r}}B(-s,s+r),\\
     \intertext{where $B$ is the beta function. By a recurrence relation for beta functions and the reflection formula, we have}
     \int_0^{\infty}\frac{w^{-s-1}}{(w+n)^r}dw&=\frac{1}{n^{s+r}}\prod_{\ell=1}^{r-1} \frac{s+r-\ell}{r-\ell}B(-s,s+1)\\
     &=-\frac{\pi}{\sin(\pi s)}\frac{1}{n^{s+r}}\prod_{\ell=1}^{r-1} \frac{s+r-\ell}{r-\ell}.
 \end{align*}
 This proves the lemma.
  \end{proof}
We next consider, as an example, the case of $\lambda=(2,1)$ and show that the function produced by our calculation interpolates $\mathcal{O}$-sum with respect to $\lambda=(2,1)$.
In $-1<\Re(s)<0$, by arranging the order of the running indices $n_{11},n_{12}$ and $n_{21}$, we compute
\begin{align*}
    &\sum_{(n_{ij})\in SSYT_{\lambda}}\frac{1}{n_{11}^{k_{11}-1}n_{12}^{k_{12}-1}n_{21}^{k_{21}-1}}\int_0^{\infty}\frac{w^{-s-1}}{(w+n_{11})(w+n_{12})(w+n_{21})}dw\\
    &=\sum_{n_{11}=n_{12}<n_{21}}\frac{1}{n_{11}^{k_{11}-1}n_{12}^{k_{12}-1}n_{21}^{k_{21}-1}}\int_0^{\infty}\frac{w^{-s-1}}{(w+n_{11})^2(w+n_{21})}dw\\
    &\hspace{5mm}+\sum_{n_{11}<n_{12}=n_{21}}\frac{1}{n_{11}^{k_{11}-1}n_{12}^{k_{12}-1}n_{21}^{k_{21}-1}}\int_0^{\infty}\frac{w^{-s-1}}{(w+n_{11})(w+n_{21})^2}dw\\
    &\hspace{5mm}+\left(\sum_{n_{11}<n_{12}<n_{21}}+\sum_{n_{11}<n_{21}<n_{12}}\right)\frac{1}{n_{11}^{k_{11}-1}n_{12}^{k_{12}-1}n_{21}^{k_{21}-1}}\int_0^{\infty}\frac{w^{-s-1}}{(w+n_{11})(w+n_{12})(w+n_{21})}dw\\
    &=\sum_{n_{11}=n_{12}<n_{21}}\frac{1}{n_{11}^{k_{11}-1}n_{12}^{k_{12}-1}n_{21}^{k_{21}-1}}\int_0^{\infty}\frac{w^{-s-1}}{(w+n_{11})^2(w+n_{21})}dw\\
     &\hspace{5mm}+\sum_{n_{11}<n_{12}=n_{21}}\frac{1}{n_{11}^{k_{11}-1}n_{12}^{k_{12}-1}n_{21}^{k_{21}-1}}\int_0^{\infty}\frac{w^{-s-1}}{(w+n_{11})(w+n_{21})^2}dw\\
    &\hspace{5mm}-\frac{\pi}{\sin(\pi s)}\left(\sum_{n_{11}<n_{12}<n_{21}}+\sum_{n_{11}<n_{21}<n_{12}}\right)\sum_{(i,j)\in D_{\lambda}}\frac{1}{n_{11}^{k_{11}}n_{12}^{k_{12}}n_{21}^{k_{21}}}
\frac{1}{n_{ij}^{s}} \prod_{(i,j)\ne (i',j')} \frac{n_{i'j'}}{ n_{i'j'}-n_{ij} }.
\end{align*}
The second and fourth terms are obtained by the same procedure as in \cite{ko}. We consider the integral 
\[\int_0^{\infty}\frac{w^{-s-1}}{(w+n_{11})^2(w+n_{21})}dw.\]
The partial fraction decomposition and Lemma \ref{betas} lead to
\begin{align*}
    &\int_0^{\infty}\frac{w^{-s-1}}{(w+n_{11})^2(w+n_{21})}dw\\
    &=\int_0^{\infty}\frac{w^{-s-1}}{(w+n_{11})^2}\frac{1}{(n_{21}-n_{11})}-\frac{w^{-s-1}}{w+n_{11}}\frac{1}{(n_{21}-n_{11})^2}+\frac{w^{-s-1}}{w+n_{21}}\frac{1}{(n_{11}-n_{21})^2}dw\\
    &=-\frac{\pi}{\sin(\pi s)}\left(\frac{(1+s)}{n_{11}^{s+2}}\frac{1}{(n_{21}-n_{11})}-\frac{1}{n_{11}^{s+1}}\frac{1}{(n_{21}-n_{11})^2}+\frac{1}{n_{21}^{s+1}}\frac{1}{(n_{11}-n_{21})^2}\right).\\
\end{align*}
In addition, we have 
\begin{align*}
    &-\frac{\sin(\pi s)}{\pi}\sum_{n_{11}=n_{12}<n_{21}}\frac{1}{n_{11}^{k_{11}-1}n_{12}^{k_{12}-1}n_{21}^{k_{21}-1}}\int_0^{\infty}\frac{w^{-s-1}}{(w+n_{11})^2(w+n_{21})}dw\\
    &=(1+s)\sum_{n_{11}=n_{12}<n_{21}}\frac{1}{n_{11}^{k_{11}}n_{12}^{k_{12}}n_{21}^{k_{21}}}\frac{1}{n_{11}^{s}}\frac{n_{21}}{(n_{21}-n_{11})}-\sum_{n_{11}=n_{12}<n_{21}}\frac{1}{n_{11}^{k_{11}}n_{12}^{k_{12}}n_{21}^{k_{21}}}\frac{1}{n_{11}^{s}}\frac{n_{11}n_{21}}{(n_{21}-n_{11})^2}\\
    &\hspace{5mm}+\sum_{n_{11}=n_{12}<n_{21}}\frac{1}{n_{11}^{k_{11}}n_{12}^{k_{12}}n_{21}^{k_{21}}}\frac{1}{n_{21}^{s}}\frac{n_{11}^2}{(n_{11}-n_{21})^2}.
\end{align*}
By changing the role of $n_{11}$ and $n_{21}$ in the above, we have a similar formula for the case $n_{11}<n_{12}=n_{21}$. Combining these calculations, we have
\begin{align}
\label{def21}
    I_{\pmb k}(s)=&(1+s)\sum_{n_{11}=n_{12}<n_{21}}\frac{1}{n_{11}^{k_{11}}n_{12}^{k_{12}}n_{21}^{k_{21}}}\frac{1}{n_{11}^{s}}\frac{n_{21}}{(n_{21}-n_{11})}\nonumber\\
    &-\sum_{n_{11}=n_{12}<n_{21}}\frac{1}{n_{11}^{k_{11}}n_{12}^{k_{12}}n_{21}^{k_{21}}}\frac{1}{n_{11}^{s}}\frac{n_{11}n_{21}}{(n_{21}-n_{11})^2}+\sum_{n_{11}=n_{12}<n_{21}}\frac{1}{n_{11}^{k_{11}}n_{12}^{k_{12}}n_{21}^{k_{21}}}\frac{1}{n_{21}^{s}}\frac{n_{11}^2}{(n_{11}-n_{21})^2}\nonumber\\
    &+(1+s)\sum_{n_{11}<n_{12}=n_{21}}\frac{1}{n_{11}^{k_{11}}n_{12}^{k_{12}}n_{21}^{k_{21}}}\frac{1}{n_{21}^{s}}\frac{n_{11}}{(n_{11}-n_{21})}\\
    &-\sum_{n_{11}<n_{12}=n_{21}}\frac{1}{n_{11}^{k_{11}}n_{12}^{k_{12}}n_{21}^{k_{21}}}\frac{1}{n_{21}^{s}}\frac{n_{11}n_{21}}{(n_{11}-n_{21})^2}+\sum_{n_{11}<n_{12}=n_{21}}\frac{1}{n_{11}^{k_{11}}n_{12}^{k_{12}}n_{21}^{k_{21}}}\frac{1}{n_{11}^{s}}\frac{n_{21}^2}{(n_{21}-n_{11})^2}\nonumber\\
    &+\left(\sum_{n_{11}<n_{12}<n_{21}}+\sum_{n_{11}<n_{21}<n_{12}}\right)\sum_{(i,j)\in D_{\lambda}}\frac{1}{n_{11}^{k_{11}}n_{12}^{k_{12}}n_{21}^{k_{21}}}
\frac{1}{n_{ij}^{s}} \prod_{(i,j)\ne (i',j')} \frac{n_{i'j'}}{ n_{i'j'}-n_{ij} }.\nonumber
\end{align}
Since we expect this $I_{\pmb k}(s)$ to interpolate $\mathcal{O}$-sum for $\lambda=(2,1)$, we substitute non-negative integers for $s$. At this stage, although non-negative integers are outside of the domain of $I_{\pmb k}(s)$ given by the integral, we can consider $I_{\pmb k}(s)$ to be analytically continued to the half plane $\Re(s)>-1$ since the series on the right-hand side converges.
Substituting $s=0$, the right-hand side becomes
\begin{align*}
    I_{\pmb k}(0)=&\sum_{n_{11}=n_{12}<n_{21}}\frac{1}{n_{11}^{k_{11}}n_{12}^{k_{12}}n_{21}^{k_{21}}}\frac{n_{21}}{(n_{21}-n_{11})}-\sum_{n_{11}=n_{12}<n_{21}}\frac{1}{n_{11}^{k_{11}}n_{12}^{k_{12}}n_{21}^{k_{21}}}\frac{n_{11}n_{21}}{(n_{21}-n_{11})^2}\\
    &+\sum_{n_{11}=n_{12}<n_{21}}\frac{1}{n_{11}^{k_{11}}n_{12}^{k_{12}}n_{21}^{k_{21}}}\frac{n_{11}^2}{(n_{11}-n_{21})^2}\\
    &+\sum_{n_{11}<n_{12}=n_{21}}\frac{1}{n_{11}^{k_{11}}n_{12}^{k_{12}}n_{21}^{k_{21}}}\frac{n_{11}}{(n_{11}-n_{21})}-\sum_{n_{11}<n_{12}=n_{21}}\frac{1}{n_{11}^{k_{11}}n_{12}^{k_{12}}n_{21}^{k_{21}}}\frac{n_{11}n_{21}}{(n_{11}-n_{21})^2}\\
    &+\sum_{n_{11}<n_{12}=n_{21}}\frac{1}{n_{11}^{k_{11}}n_{12}^{k_{12}}n_{21}^{k_{21}}}\frac{n_{21}^2}{(n_{21}-n_{11})^2}\\
    &+\sum_{n_{11}<n_{21}<n_{12}}\frac{1}{n_{11}^{k_{11}}n_{12}^{k_{12}}n_{21}^{k_{21}}}+\sum_{n_{11}<n_{12}<n_{21}}\frac{1}{n_{11}^{k_{11}}n_{12}^{k_{12}}n_{21}^{k_{21}}}\\
    =&\zeta_\lambda(\pmb k).
  \end{align*}
  Substituting $s=m\in\mathbb Z_{\ge0}$, the right-hand side becomes
\begin{align*}
    I_{\pmb k}(m)=&(1+m)\sum_{n_{11}=n_{12}<n_{21}}\frac{1}{n_{11}^{k_{11}}n_{12}^{k_{12}}n_{21}^{k_{21}}}\frac{1}{n_{11}^{m}}\frac{n_{21}}{(n_{21}-n_{11})}\\
    &-\sum_{n_{11}=n_{12}<n_{21}}\frac{1}{n_{11}^{k_{11}}n_{12}^{k_{12}}n_{21}^{k_{21}}}\frac{1}{n_{11}^{m}}\frac{n_{11}n_{21}}{(n_{21}-n_{11})^2}+\sum_{n_{11}=n_{12}<n_{21}}\frac{1}{n_{11}^{k_{11}}n_{12}^{k_{12}}n_{21}^{k_{21}}}\frac{1}{n_{21}^{m}}\frac{n_{11}^2}{(n_{11}-n_{21})^2}\\
    &+(1+m)\sum_{n_{11}<n_{12}=n_{21}}\frac{1}{n_{11}^{k_{11}}n_{12}^{k_{12}}n_{21}^{k_{21}}}\frac{1}{n_{21}^{m}}\frac{n_{11}}{(n_{11}-n_{21})}\\
    &-\sum_{n_{11}<n_{12}=n_{21}}\frac{1}{n_{11}^{k_{11}}n_{12}^{k_{12}}n_{21}^{k_{21}}}\frac{1}{n_{21}^{m}}\frac{n_{11}n_{21}}{(n_{11}-n_{21})^2}+\sum_{n_{11}<n_{12}=n_{21}}\frac{1}{n_{11}^{k_{11}}n_{12}^{k_{12}}n_{21}^{k_{21}}}\frac{1}{n_{11}^{m}}\frac{n_{21}^2}{(n_{21}-n_{11})^2}\\
    &+\sum_{e_1+e_2+e_3=m}\sum_{n_{11}<n_{21}<n_{12}}\frac{1}{n_{11}^{k_{11}+e_1}n_{12}^{k_{12}+e_2}n_{21}^{k_{21}+e_3}}\\
    &+\sum_{e_1+e_2+e_3=m}\sum_{n_{11}<n_{12}<n_{21}}\frac{1}{n_{11}^{k_{11}+e_1}n_{12}^{k_{12}+e_2}n_{21}^{k_{21}+e_3}}\\
    =&\sum_{n_{11}=n_{12}<n_{21}}\frac{1}{n_{11}^{k_{11}}n_{12}^{k_{12}}n_{21}^{k_{21}}}\left(\frac{m+1}{n_{11}^m}+\frac{m}{n_{11}^{m-1}n_{21}}+\cdots+\frac{1}{n_{21}^m}\right)\\
    &+\sum_{n_{11}<n_{12}=n_{21}}\frac{1}{n_{11}^{k_{11}}n_{12}^{k_{12}}n_{21}^{k_{21}}}\left(\frac{m+1}{n_{21}^m}+\frac{m}{n_{21}^{m-1}n_{11}}+\cdots+\frac{1}{n_{11}^m}\right)\\
    &+\left(\sum_{n_{11}<n_{21}<n_{12}}+\sum_{n_{11}<n_{12}<n_{21}}\right)\sum_{e_1+e_2+e_3=m}\frac{1}{n_{11}^{k_{11}+e_1}n_{12}^{k_{12}+e_2}n_{21}^{k_{21}+e_3}}\\
    =&\sum_{|\pmb \varepsilon|=m}\zeta_{\lambda}(\pmb k+\pmb \varepsilon).
  \end{align*}
The above two calculations ensure that our $I_{\pmb k}(s)$ interpolates $\mathcal{O}$-sum associated with $\lambda=(2,1)$.
Based on this, we would like to produce a series expression of $I_{\pmb{k}}(s)$ for the general case, as well. In preparation for that, we offer the following lemma, which gives explicitly the coefficients of the partial fraction decomposition:
\begin{lemma}
\label{partial}
Let 
\[P(N)=\prod_{\alpha=1}^{R_N}\frac{1}{(w+n_\alpha)^{r_{\alpha}}}\]
with distinct integers $n_1,\ldots,n_{R_N}$. Then 
\[P(N)=\sum_{\alpha=1}^{R_N}\sum_{\ell=1}^{r_{\alpha}}\frac{1}{(w+n_\alpha)^{\ell}}\frac{d^{r_\alpha-\ell}}{dw^{r_\alpha-\ell}}\left.\left(\frac{1}{(r_\alpha-\ell)!}\prod_{\beta\neq\alpha}\frac{1}{(w+n_\beta)^{r_{\beta}}}\right)\right|_{w=-n_\alpha}.\]
\end{lemma}
\begin{proof}
This lemma follows from the uniqueness of the Laurent series expansion.
\end{proof}
We apply Lemma \ref{partial} with Lemma \ref{betas}, then it holds that
\begin{align}
\label{integralpart}
    &-\frac{\sin(\pi s)}{\pi}\int_0^\infty P(N)w^{-s-1}\ dw\nonumber\\
    &=\sum_{\alpha=1}^{R_N}\sum_{\ell=1}^{r_{\alpha}}\frac{1}{n_{\alpha}^{s+\ell}}\prod_{p=1}^{\ell-1} \frac{s+\ell-p}{\ell-p}\frac{d^{r_\alpha-\ell}}{dw^{r_\alpha-\ell}}\left.\left(\frac{1}{(r_\alpha-\ell)!}\prod_{\beta\neq\alpha}\frac{1}{(w+n_\beta)^{r_{\beta}}}\right)\right|_{w=-n_\alpha}.
\end{align}
For $N=(n_{ij})\in SSYT_{\lambda}$, we rewrite \[\prod_{(i,j)\in D_{\lambda/\mu}}\frac{1}{w+n_{ij}}=\prod_{\alpha=1}^{R_N}\frac{1}{(w+n_\alpha)^{r_{\alpha}}}\]
by summarizing the same $n_{ij}$.
Identity (\ref{integralpart}) then leads to
\begin{align*}
    &I_{\pmb k}(s)\\
    &=-\frac{\sin(\pi s)}{\pi}\sum_{N\in SSYT_{\lambda/\mu}}\frac{1}{N^{\pmb k-\pmb 1}}\int_0^{\infty}w^{-s-1}\prod_{(i,j)\in D_{\lambda/\mu}}\frac{1}{w+n_{ij}}dw\\
    &=\sum_{N\in SSYT_{\lambda/\mu}}\frac{1}{N^{\pmb k-\pmb 1}}\sum_{\alpha=1}^{R_N}\sum_{\ell=1}^{r_{\alpha}}\frac{1}{n_{\alpha}^{s+\ell}}\prod_{p=1}^{\ell-1} \frac{s+\ell-p}{\ell-p}\frac{d^{r_\alpha-\ell}}{dw^{r_\alpha-\ell}}\left.\left(\frac{1}{(r_\alpha-\ell)!}\prod_{\beta\neq\alpha}\frac{1}{(w+n_\beta)^{r_{\beta}}}\right)\right|_{w=-n_\alpha}.
\end{align*} 
We can now summarize the above as follows.
  \begin{lemma}[Explicit series form of $I_{\pmb k}(s)$]
  \label{explicit}
    For $-1<\Re(s)<0$,
\[    I_{\pmb k}(s)=\sum_{N\in SSYT_{\lambda/\mu}}\frac{1}{N^{\pmb k-\pmb 1}}\sum_{\alpha=1}^{R_N}\sum_{\ell=1}^{r_{\alpha}}\frac{1}{n_{\alpha}^{s+\ell}}\prod_{p=1}^{\ell-1} \frac{s+\ell-p}{\ell-p}\frac{d^{r_\alpha-\ell}}{dw^{r_\alpha-\ell}}\left.\left(\frac{1}{(r_\alpha-\ell)!}\prod_{\beta\neq\alpha}\frac{1}{(w+n_\beta)^{r_{\beta}}}\right)\right|_{w=-n_\alpha}.\]  
  \end{lemma}
  Substituting $s=0$, we have
  \begin{align*}
      I_{\pmb k}(0)&=\sum_{N\in SSYT_{\lambda/\mu}}\frac{1}{N^{\pmb k-\pmb 1}}\sum_{\alpha=1}^{R_N}\sum_{\ell=1}^{r_{\alpha}}\frac{1}{n_{\alpha}^{\ell}}\frac{d^{r_\alpha-\ell}}{dw^{r_\alpha-\ell}}\left.\left(\frac{1}{(r_\alpha-\ell)!}\prod_{\beta\neq\alpha}\frac{1}{(w+n_\beta)^{r_{\beta}}}\right)\right|_{w=-n_\alpha}\\
      &=\sum_{N\in SSYT_{\lambda/\mu}}\frac{1}{N^{\pmb k-\pmb 1}}\prod_{\alpha=1}^{R_N}\frac{1}{n_\alpha^{r_{\alpha}}}\\
      &=\zeta_{\lambda/\mu}(\pmb k).
  \end{align*}  
This ensures that the series expansion converges in $\Re(s)\ge0$, which gives the analytic continuation of $I_{\pmb k}(s)$ in $\Re(s)>-1$. Furthermore, the series expansion given in Lemma \ref{explicit} is a sum of products of the polynomial and zeta functions associated with a root system of type $A$. Therefore, $I_{\pmb k}(s)$ can be meromorphically continued to the whole space of $\mathbb C$ (see \cite[Section 2]{mt}).
 \section{Interpolation of the generalized duality formula}
 \label{section3}
In this section, we revisit \cite[Lemma 2.1]{hmo} and generalize the lemma with the case $a_i=a_j$.
\begin{lemma}[{\cite[Lemma 2.1]{hmo}}] \label{partial_frac}
 For $m\in\mathbb{Z}_{\ge 0}$ and $a_1,\ldots,a_r\in\mathbb{R}$ with $a_i\neq a_j$ for $i\neq j$, we have
 \begin{align*}  
  \sum_{\substack{ e_1+\cdots+e_r=m \\ e_i\ge0\,(1\le i\le r) }} 
  a_1^{e_1}\cdots a_r^{e_r} 
  =\sum_{i=1}^r a_i^{m+r-1} \prod_{j\ne i}(a_i-a_j)^{-1}. 
 \end{align*} 
\end{lemma}
We note that if $a_1=a_2$, then for each $i=1,2$ the product $\prod_{j\ne i}(a_i-a_j)^{-1}$ is not defined.
On the other hand, following the way to the proof of Lemma \ref{partial_frac} in \cite{hmo}, we can obtain the partial fraction decomposition form formally. 
For example, letting
 \[
  A_i=a_i^{r-1} \prod_{j\ne i}(a_i-a_j)^{-1},
 \]
 we have
 \begin{align*}
 \frac{1}{1-a_1x}\frac{1}{1-a_2x} \frac{1}{1-a_3x}&=\frac{A_1}{1-a_1x}+\frac{A_2}{1-a_2x}+\frac{A_3}{1-a_3x}\\
 &=\frac{A_1(1-a_2x)+A_2(1-a_1x)}{(1-a_1x)(1-a_2x)}+\frac{A_3}{1-a_3x}.
 \intertext{We note that \[A_1(1-a_2x)+A_2(1-a_1x)=\frac{a_1a_2+a_1a_2a_3x-a_3(a_1+a_2)}{(a_1-a_3)(a_2-a_3)}.\] Therefore,}
 \frac{1}{1-a_1x}\frac{1}{1-a_2x} \frac{1}{1-a_3x} &=\frac{a_1a_2+a_1a_2a_3x-a_3(a_1+a_2)}{(1-a_1x)(1-a_2x)(a_1-a_3)(a_2-a_3)}+\frac{A_3}{1-a_3x}.
 \intertext{Substituting $a_1=a_2$, we have}
 \frac{1}{(1-a_1x)^2} \frac{1}{1-a_3x}&=\frac{a_1^2+a_1^2a_3x-2a_3a_1}{(1-a_1x)^2(a_1-a_3)^2}+\frac{A_3}{1-a_3x}\\
 &=\frac{a_1}{(1-a_1x)^2(a_1-a_3)}-\frac{a_1a_3}{(1-a_1x)(a_1-a_3)^2}+\frac{A_3}{1-a_3x}.
 \end{align*}
Using the above identity, we have
\begin{equation}
\label{ue}
    \sum_{\substack{ e_1+e_2+e_3=m \\ e_i\ge0\,(1\le i\le 3) }} 
  a_1^{e_1}a_2^{e_2} a_3^{e_3} 
  =(m+1)a_1^{m}\frac{a_1}{a_1-a_3}+a_1^{m}\frac{a_1a_3}{(a_1-a_3)^2}+a_3^{m}A_3. 
\end{equation}   
Substituting $a_1=n_{11}^{-1}, a_2=n_{12}^{-1}$ and $a_3=n_{21}^{-1}$ into (\ref{ue}) and making a simple calculation, we obtain formula (\ref{def21}) interpolating $\mathcal{O}$-sum for Schur multiple zeta values of shape $(2,1)$.
Indeed, keeping $a_1=a_2$ and $n_{11}=n_{12}$ in mind, we have
\begin{align}
\label{comp}
&\sum_{\substack{ e_1+e_2+e_3=m \\ e_i\ge0\,(1\le i\le 3) }} 
  \frac{1}{n_{11}^{e_1}}\frac{1}{n_{12}^{e_2}}\frac{1}{n_{21}^{e_3}}\nonumber\\
  &=(m+1)\frac{1}{n_{11}^{m+1}}\left(\frac{1}{n_{11}}-\frac{1}{n_{21}}\right)^{-1}+\frac{1}{n_{11}^{m+1}}\frac{1}{n_{21}}\left(\frac{1}{n_{11}}-\frac{1}{n_{21}}\right)^{-2}+\frac{1}{n_{21}^{m+2
  }}\left(\frac{1}{n_{21}}-\frac{1}{n_{11}}\right)^{-2}\nonumber\\
  &=(m+1)\frac{1}{n_{11}^{m}}\frac{n_{21}}{n_{21}-n_{11}}+\frac{1}{n_{11}^{m}}\frac{n_{11}n_{21}}{(n_{21}-n_{11})^2}+\frac{1}{n_{21}^{m
  }}\frac{n_{11}^2}{(n_{11}-n_{21})^2}.
\end{align}
This calculation corresponds to the terms of (\ref{def21}) with $n_{11}=n_{12}$.
Swapping the role of $n_{11}$ and $n_{21}$, we obtain the identity corresponding to $n_{12}=n_{21}$.
Thus, it holds that
\begin{align*}
    I_{\pmb k}(m)&=\left(\sum_{n_{11}<n_{12}<n_{21}}+\sum_{n_{11}<n_{21}<n_{12}}\right)\frac{1}{n_{11}^{k_{11}}n_{12}^{k_{12}}n_{21}^{k_{21}}}
  \sum_{\substack{ e_1+e_2+e_3=m \\ e_i\ge0\,(1\le i\le 3) }} 
  \frac{1}{n_{11}^{e_1}}\frac{1}{n_{12}^{e_2}}\frac{1}{n_{21}^{e_3}}\\
    &\hspace{5mm}+\left(\sum_{n_{11}=n_{12}<n_{21}}+\sum_{n_{11}<n_{12}=n_{21}}\right)\frac{1}{n_{11}^{k_{11}}n_{12}^{k_{12}}n_{21}^{k_{21}}}\sum_{\substack{ e_1+e_2+e_3=m \\ e_i\ge0\,(1\le i\le 3) }} 
  \frac{1}{n_{11}^{e_1}}\frac{1}{n_{12}^{e_2}}\frac{1}{n_{21}^{e_3}}\\
    &=\sum_{|\pmb \varepsilon|=m}\zeta_{\lambda}(\pmb k+\pmb \varepsilon).
\end{align*}
The first two terms are obtained by the same calculation as in \cite{hmo}; the last two are obtained via computation (\ref{comp}). 

More generally, for fixed $A=(a_i)$, by summarizing the same $a_{i}$, we can define
\[P(A)=\prod_{i=1}^r\frac{1}{1-a_ix}=\prod_{\alpha=1}^{R_A}\frac{1}{(1-b_\alpha x)^{r_{\alpha}}}\]
with distinct $b_1,\ldots,b_{R_A}$. Then, as in the proof of Lemma \ref{partial}, the uniqueness of the Laurent series expansion gives
\[P(A)=\sum_{\alpha=1}^{R_A}\sum_{\ell=1}^{r_{\alpha}}\frac{1}{(1-b_\alpha x)^{\ell}}\frac{d^{r_\alpha-\ell}}{dx^{r_\alpha-\ell}}\left.\left(\frac{1}{(-b_\alpha)^{r_\alpha-\ell}(r_\alpha-\ell)!}\prod_{\beta\neq\alpha}\frac{1}{(1-b_\beta x)^{r_{\beta}}}\right)\right|_{x=b_\alpha^{-1}}.\]
By expanding into the geometric series, we have the following Lemma:
\begin{lemma}
\label{lem21ver}
\[\sum_{\substack{ e_1+\cdots+e_r=m \\ e_i\ge0\,(1\le i\le r) }} 
  a_1^{e_1}\cdots a_r^{e_r} 
  =\sum_{\alpha=1}^{R_A}\sum_{\ell=1}^{r_{\alpha}}\binom{m+\ell-1}{\ell-1}b_\alpha^{m}A_\alpha^{(\ell)},\]
  where \[A_\alpha^{(\ell)}=\frac{d^{r_\alpha-\ell}}{dx^{r_\alpha-\ell}}\left.\left(\frac{1}{(-b_\alpha)^{r_\alpha-\ell}(r_\alpha-\ell)!}\prod_{\beta\neq\alpha}\frac{1}{(1-b_\beta x)^{r_{\beta}}}\right)\right|_{x=b_\alpha^{-1}}.\] 
\end{lemma}
\begin{theorem}
  \label{integralrepschur}
  The function $I_{\pmb k}(s)$, defined by
  \begin{align*}
   I_{\pmb k}(s)&=-\frac{\sin(\pi s)}{\pi}\sum_{(n_{ij})\in SSYT_{\lambda/\mu}}\prod_{(i,j)\in D_{\lambda/\mu}}\frac{1}{n_{ij}^{k_{ij}-1}}\int_0^{\infty}w^{-s-1}\prod_{(i,j)\in D_{\lambda/\mu}}\frac{1}{w+n_{ij}}dw,
   \intertext{or}
   I_{\pmb k}(s)&=\sum_{N\in SSYT_{\lambda/\mu}}\frac{1}{N^{\pmb k-\pmb 1}}\sum_{\alpha=1}^{R_N}\sum_{\ell=1}^{r_{\alpha}}\frac{1}{n_{\alpha}^{m+\ell}}\prod_{p=1}^{\ell-1} \frac{s+\ell-p}{\ell-p}\frac{d^{r_\alpha-\ell}}{dw^{r_\alpha-\ell}}\left.\left(\frac{1}{(r_\alpha-\ell)!}\prod_{\beta\neq\alpha}\frac{1}{(w+n_\beta)^{r_{\beta}}}\right)\right|_{w=-n_\alpha},  
  \end{align*}
interpolates $\mathcal{O}$-sum for the Schur multiple zeta values of shape $\lambda/\mu$.
\end{theorem}
\begin{proof}
  As above, substituting $b_{\alpha}=n_{\alpha}^{-1}$ into Lemma \ref{lem21ver}, we can compute
  \begin{align*}
  &\sum_{\substack{ \sum e_{ij}=m \\ e_{ij}\ge0\ }} 
  \prod_{(i,j)\in D_{\lambda/\mu}}\left(\frac{1}{n_{ij}}\right)^{e_{ij}}\\
  &=\sum_{\alpha=1}^{R_N}\sum_{\ell=1}^{r_{\alpha}}\binom{m+\ell-1}{\ell-1}\frac{1}{n_\alpha^{m}}\frac{d^{r_\alpha-\ell}}{dx^{r_\alpha-\ell}}\left.\left((-n_\alpha)^{r_\alpha-\ell}\frac{1}{(r_\alpha-\ell)!}\prod_{\beta\neq\alpha}\frac{1}{(1-n_\beta^{-1} x)^{r_{\beta}}}\right)\right|_{x=n_\alpha}\\
  &=N^{\pmb 1}\sum_{\alpha=1}^{R_N}\sum_{\ell=1}^{r_{\alpha}}\binom{m+\ell-1}{\ell-1}\frac{1}{n_\alpha^{m+\ell}}\frac{d^{r_\alpha-\ell}}{dx^{r_\alpha-\ell}}\left.\left((-1)^{r_\alpha-\ell}\frac{1}{(r_\alpha-\ell)!}\prod_{\beta\neq\alpha}\frac{1}{(n_{\beta}- x)^{r_{\beta}}}\right)\right|_{x=n_\alpha},
  \intertext{where $N^{\pmb 1}=\prod n_{ij}=\prod n_{\alpha}^{r_\alpha}$. Changing the variable by $w=-x$, we obtain}
  &N^{\pmb 1}\sum_{\alpha=1}^{R_N}\sum_{\ell=1}^{r_{\alpha}}\binom{m+\ell-1}{\ell-1}\frac{1}{n_\alpha^{m+\ell}}\frac{d^{r_\alpha-\ell}}{dx^{r_\alpha-\ell}}\left.\left(\frac{1}{(r_\alpha-\ell)!}\prod_{\beta\neq\alpha}\frac{1}{(n_{\beta}+ w)^{r_{\beta}}}\right)\right|_{w=-n_\alpha}.
  \end{align*}
Therefore, we have  
  \begin{align*}
      I_{\pmb k}(m)&=\sum_{N\in SSYT_{\lambda/\mu}}\frac{1}{N^{\pmb k-\pmb 1}}\sum_{\alpha=1}^{R_N}\sum_{\ell=1}^{r_{\alpha}}\frac{1}{n_{\alpha}^{m+\ell}}\prod_{p=1}^{\ell-1} \frac{m+\ell-p}{\ell-p}\frac{d^{r_\alpha-\ell}}{dw^{r_\alpha-\ell}}\left.\left(\frac{1}{(r_\alpha-\ell)!}\prod_{\beta\neq\alpha}\frac{1}{(w+n_\beta)^{r_{\beta}}}\right)\right|_{w=-n_\alpha}\\
      &=\sum_{N\in SSYT_{\lambda/\mu}}\frac{1}{N^{\pmb k-\pmb 1}}\sum_{\alpha=1}^{R_N}\sum_{\ell=1}^{r_{\alpha}}\frac{1}{n_{\alpha}^{m+\ell}}\binom{m+\ell-1}{\ell-1}\frac{d^{r_\alpha-\ell}}{dw^{r_\alpha-\ell}}\left.\left(\frac{1}{(r_\alpha-\ell)!}\prod_{\beta\neq\alpha}\frac{1}{(w+n_\beta)^{r_{\beta}}}\right)\right|_{w=-n_\alpha}\\
      &=\sum_{N\in SSYT_{\lambda/\mu}}\frac{1}{N^{\pmb k}}\sum_{\substack{ \sum e_{ij}=m \\ e_{ij}\ge0\ }} 
  \prod_{(i,j)\in D_{\lambda/\mu}}\left(\frac{1}{n_{ij}}\right)^{e_{ij}}.
  \end{align*}  
  This ensures that $I_{\pmb k}(s)$ interpolates the Ohno {\color{blue}functions} to the Schur multiple zeta functions. 
\end{proof}  
Finally, we obtain the duality formula for $I_{\pmb k}(s)$ as complex functions:
\begin{theorem}
\label{dualityschur}
Let $\lambda$ and $\mu$ be two partitions such that $\lambda_i\ge\mu_i$ for all $i$, and let $\delta=\lambda/\mu$. Let ${\pmb k}^{\dagger}$ be the dual tableau of ${\pmb k}\in I_\delta^D$. Then, for $s\in \mathbb C$ we have
\[
I_{\pmb k}(s)=I_{{\pmb k}^{\dagger}}(s).
\]
\end{theorem}
\begin{proof}
 By Theorem \ref{Thmohno} and Theorem \ref{integralrepschur}, we have 
 $I_{\pmb{k}}(s)=I_{\pmb{k}^\dagger}(s)$ for $s\in\mathbb{Z}_{\ge 0}$.
 As
 \[
 \sum_{N\in SSYT_{\delta}}\frac{1}{N^{\pmb k-\pmb 1}}\left(\sum_{\alpha=1}^{R_N}\sum_{\ell=1}^{r_{\alpha}}\frac{1}{n_{\alpha}^{m+\ell}}\binom{m+\ell-1}{\ell-1}\frac{d^{r_\alpha-\ell}}{dw^{r_\alpha-\ell}}\left.\left(\frac{1}{(r_\alpha-\ell)!}\prod_{\beta\neq\alpha}\frac{1}{(w+n_\beta)^{r_{\beta}}}\right)\right|_{w=-n_\alpha}\right)
 \]
 is a Dirichlet series for each $(i,j)\in D_{\lambda/\mu}$, the function $I_{\pmb{k}}(s)$ is also a Dirichlet series.  
 By the uniqueness theorem for the Dirichlet series (see \cite[Theorem 11.3]{ap}), we have
 $I_{\pmb{k}}(s)=I_{\pmb{k}^\dagger}(s)$ for $\Re(s)> -1$ and ${\pmb k} \in I_{\delta}^{\rm D}$. Moreover, $I_{\pmb k}(s)$ is meromorphically continued to the whole space of $\mathbb C$. Thus, the assertion is proved.
\end{proof}
{\color{blue}\begin{remark}
When we put $\lambda=(\{1\}^r)$ and $\mu=\emptyset$ for a positive integer $r$, then Theorem \ref{dualityschur} implies Theorem \ref{interpolation}. Furthermore, substituting non-negative integers for $s$, we obtain Theorem \ref{ohno}.
\end{remark}
}
\section*{Acknowledgement}
This work was supported by
Grant-in-Aid for Scientific Research (A) (Grant Number: JP21H04430), 
Grant-in-Aid for Scientific Research (B) (Grant Number: JP18H01110),
Grant-in-Aid for Scientific Research (C) (Grant Number: JP18K03223, JP19K03437 and JP22K03274) and Grant-in-Aid for Early-Career Scientists (Grant Number: JP22K13900).
This work was supported by the Research Institute for Mathematical Sciences,
an International Joint Usage/Research Center located in Kyoto University.

\end{document}